\documentclass[12pt]{article}
\usepackage[a4paper,margin=1.1in]{geometry}
\usepackage{amssymb}
\usepackage{amsmath}
\usepackage{amsfonts}
\usepackage{amsthm}
\usepackage{color}
\usepackage{float}
\usepackage[all]{xy}
\usepackage{scalerel}
\usepackage{mathabx}
\usepackage{delimset}

\newcommand{\simplinomial}{\genfrac{\langle}{\rangle}{0pt}{}}
\newcommand{\stirling}{\genfrac{[}{]}{0pt}{}}
\newcommand{\stirlingII}{\genfrac{\{}{\}}{0pt}{}}

\newcommand{\C}{\mathbb{C}}

\newcommand{\N}{\mathbb{N}}

\newcommand{\R}{\mathbb{R}}

\newcommand{\D}{\mathrm{D}}
\newcommand{\s}{\mathrm{S}}

\newcommand{\e}{\mathrm{e}}

\newtheorem{theorem}{Theorem}

\newtheorem{definition}{Definition}
\newtheorem{proposition}{Proposition}
\newtheorem{corollary}{Corollary}
\newtheorem{example}{Example}

\begin{document}

\title{\textbf{Hypergeometric Bernoulli Polynomials Defined on Simplicial $d$-Polytopic Numbers}}
\author{Ronald Orozco L\'opez}

\newcommand{\Addresses}{{
  \bigskip
  \footnotesize

  \textit{E-mail address}, R.~Orozco: \texttt{rj.orozco@uniandes.edu.co}
  
}}

\maketitle
\tableofcontents

\begin{abstract}
We introduce an $\s_d$-analogue of the hypergeometric Bernoulli polynomials and study their properties. To achieve this goal, we introduce a calculus defined on the simplicial $d$-polytopic numbers. Two definitions of the $\s_d$-derivatives are given. These two definitions allow us to derive an identity relating Kummer confluent hypergeometric function and Touchard polynomials. This calculus is closely related to the $d$-Hoggatt binomial coefficients. $\s_d$-analogs of the exponential function and the hypergeometric functions are given.
\end{abstract}

\section{Introduction}

The Bernoulli polynomials $B_n(x)$, named after Jacob Bernoulli and defined by the generating function
\begin{equation}
\frac{te^x}{e^t-1}=\sum_{n=0}^{\infty}B_{n}\frac{x^n}{n!},\hspace{0.5cm}\vert t\vert<2\pi,
\end{equation}
where $B_n=B_{n}(0)$ are the Bernoulli numbers, are quite likely the most widely studied polynomials. They play a fundamental role in combinatorics, number theory, numerical analysis, approximation theory, differential equations, and other fields.
Some of the most important properties of Bernoulli polynomials $B_n(x)$ are:
\begin{align*}
    &B_{n}(0)=B_{n}(1)=B_n,\hspace{0.5cm}n\neq1,\\
    B_{n}(x+1)-B_{n}(x)&=nx^{n-1},\hspace{1.5cm}B^{\prime}_n(x)=bB_{n-1}(x),\\
    B_{n}(x)&=\sum_{k=0}^{n}\binom{n}{k}B_kx^{n-k}.
\end{align*}
Many generalizations of these polynomials have been studied. Among them, the best known are:
The generalized Bernoulli polynomials defined by Gatteschi \cite{gattes} for all $n\in\N$ and $\alpha\in\C$,
\begin{equation}
    e^{tx}\left(\frac{t}{e^t-1}\right)^\alpha=\sum_{n=0}^{\infty}B_{n}^{(\alpha)}(x)\frac{t^n}{n!}.
\end{equation}
The Apostol-Bernoulli polynomials studied by Luo \cite{luo} 
\begin{equation}
    \frac{te^{tx}}{\lambda e^t-1}=\sum_{n=0}^{\infty}B_{n}(x;\lambda)\frac{t^n}{n!},\hspace{0.5cm}\vert t+\ln\lambda\vert<2\pi.
\end{equation}
Other Apostol-Bernoulli polynomials with two parameters were studied by Srivastava et al. \cite{srivastava}.

The hypergeometric Bernoulli polynomials defined by Howard \cite{howard1,howard2}
\begin{equation}\label{hbp}
    \frac{t^me^{tx}/m!}{e^t-\mathrm{T}_{m-1}(t)}=\sum_{n=0}^{\infty}B_{n}(m,x)\frac{t^n}{n!},
\end{equation}
where $\mathrm{T}_{m-1}(x)$ is the Taylor polynomial of order $m-1$ for the exponential function.

Other generalizations are provided by the $q$-analogs and Fibo-analogs of the Bernoulli polynomials. Carlitz \cite{carlitz1,carlitz2} defined and studied the $q$-Bernoulli numbers and polynomials. A great deal of research has been devoted to these $q$-analogs \cite{al1}, \cite{al2}, \cite{luo2}, \cite{man}, \cite{duran1}, \cite{duran2}. A $q$-analogue of the hypergeometric Bernoulli polynomials Eq.(\ref{hbp}) was defined by Sadjang et al
\begin{equation}
    \frac{t^m/\brk[s]{m}_q!}{e_q(t)-\mathrm{T}_{m-1,q}(t)}e_q(tx)=\sum_{n=0}^{\infty}B_{n,q}(m,x)\frac{t^n}{\brk[s]{n}_q!},
\end{equation}
where $\brk[s]{n}_q=\frac{1-q^n}{1-q}$ is the $q$-number, $\brk[s]{n}_d!=\brk[s]{1}_d\brk[s]{2}_d\cdots\brk[s]{n}_d$, $e_q(t)$ is the $q$-exponential function
\begin{equation}
    e_q(t)=\sum_{n=0}^{\infty}\frac{t^n}{\brk[s]{n}_q!}
\end{equation}
and
\begin{equation}
    \mathrm{T}_{m,q}(t)=\sum_{n=0}^{m}\frac{t^n}{\brk[s]{n}_q!}.
\end{equation}

The Bernoulli-Fibonacci polynomials $B_{n}^F(x)$ defined by Pashaev \cite{pashaev}
\begin{equation}
    \frac{te_{F}^{tx}}{e_{F}^t-1}=\sum_{n=0}^{\infty}B_{n}^{F}(x;\lambda)\frac{t^n}{F_n!},
\end{equation}
where $F_n$ are the Fibonacci numbers, $F_n!=F_1F_2\cdots F_n$ and $e_F^x$ is the Golden exponential function
\begin{equation}
    e_F^x=\sum_{n=0}^{\infty}\frac{t^n}{F_n!}.
\end{equation}
Also, the Apostol Bernoulli-Fibonacci polynomials of order $\alpha$ \cite{tuglu} defined by
\begin{equation}
    \left(\frac{t}{\lambda e_{F}^t-1}\right)^\alpha e_{F}^{tx}=\sum_{n=0}^{\infty}B_{n,F}^{\alpha}(x;\lambda)\frac{t^n}{F_n!}.
\end{equation}
Other generalizations based on the Fibonacci numbers can be found in \cite{can}, \cite{kus}.

In this article, we introduce the $\s_d$-hypergeometric Bernoulli polynomials $B_{d,n}(m;x)$ by the following generating function
\begin{equation}
    \frac{t^m/\brk[a]{m}_d!}{\exp_{d}(t)-\mathrm{T}_{d,m-1}(t)}\exp_{d}(xt)=\frac{\exp_{d}(xt)}{{}_{1}\sigma^d_{1}(1;m+1;t)}=\sum_{n=0}^{\infty}B_{d,n}(m;x)\frac{t^n}{\brk[a]{n}_{d}!},
\end{equation}
where $\brk[a]{n}_d$ is the sequence of simplicial $d$-polytope numbers
\begin{equation}
    \brk[a]{n}_d=\binom{n+d-1}{d},
\end{equation}
with $\brk[a]{n}_d!=\brk[a]{1}_d\brk[a]{2}_d\cdots\brk[a]{n}_d$, $\exp_d(x)$ is the $\s_d$-exponential function
\begin{align}
    \exp_d(x)&=\prod_{i=0}^{d-1}i!\sum_{n=0}^{\infty}\frac{(d!x)^n}{n!(n+1)!(n+2)!\cdots(n+d-1)!},\\
    \mathrm{T}_{d,m}(t)&=\prod_{i=0}^{d-1}i!\sum_{n=0}^{m}\frac{(d!x)^n}{n!(n+1)!(n+2)!\cdots(n+d-1)!},
\end{align}
and ${}_{1}\sigma_{1}^d(a;b;t)$ is the $\s_d$-analog of the Kummer confluent hypergeometric function.
The polynomials $B_{d,n}(m;x)$ fulfill the translation formula
\begin{equation}
    B_{d,n}(m;x\oplus_dy)=\sum_{k=0}^{n}\simplinomial{n}{k}_{d}B_{d,n}(m;x)y^{n-k}    
\end{equation}
where
\begin{equation}
    (x\oplus_dy)^{(n)}=\sum_{k=0}^{n}\simplinomial{n}{k}_dx^{n-k}y^k
\end{equation}
are the bivariate $d$-Hoggatt polynomials, which are a generalization of $d$-Hoggatt polynomials \cite{cigler2}, and $\simplinomial{n}{k}_d$
is the $d$-Hoggatt binomial coefficients \cite{cigler}
\begin{equation}
    \simplinomial{n}{k}_d=\frac{\brk[a]{n}_d!}{\brk[a]{k}_d!\brk[a]{n-k}_d!}.
\end{equation}
The framework we have used to discuss $\s_d$-hypergeometric Bernoulli polynomials is a calculus defined on $d$-simplicial polytopic numbers, where the $\s_d$-derivative operator is defined by
\begin{equation}\label{der1}
    \D_{S_{d}}f(x)=\sum_{k=0}^{d-1}\binom{d-1}{k}\frac{1}{(k+1)!}X^{k}\D^{k+1}f(x)
\end{equation}
or
\begin{equation}\label{der2}
        \D_{\s^\prime_d}=\frac{1}{d!}\sum_{k=0}^{d}\stirling{d}{k}X^{-1}(XD)^k,
    \end{equation}
where $\stirling{n}{k}$ are the numbers of Stirling of the first kind, $D=\frac{d}{dx}$ is the derivative operator, and $f(x)\in C^{k+1}$. For $d$ ranging from $1$ to $5$, we have the following $\s_d$-derivatives, respectively: the ordinary derivative $\frac{d}{dx}$, the triangular derivative $\D_\mathrm{T}$, the tetrahedral derivative $\D_\mathrm{Te}$, the pentachoron derivative $\D_\mathrm{P}$, and the hexateron derivative $\D_\mathrm{H}$. By using the operator Eqs. (\ref{der1}) and (\ref{der2}) we obtain a beautiful relation between the Kummer confluent hypergeometric function $_{1}F_{1}(a;b;x)$ and the Touchard polynomials $\mathrm{T}_n(x)$ of degree $n$
\begin{equation}
    {}_{1}F_{1}(1-d;2;-x)=\frac{x^{-1}}{d!}\sum_{k=0}^{d}\stirling{d}{k}\mathrm{T}_k(x).
\end{equation}
The function $\exp_d(x)$ is solution of the differential equation
\begin{equation}
    \sum_{k=0}^{d-1}\binom{d-1}{k}\frac{1}{(k+1)!}x^{k}y^{(k+1)}=y,
\end{equation}
in this regard, the function $\exp_d(x)$ is an eigenfunction of the operator $\D_{\s_d}$.

\section{Simplicial polytopic numbers}

The simplicial polytopic numbers are a family of sequences of figurate numbers corresponding to the $d$-dimensional simplex for each dimension $d$, where $d$ is a non-negative integer. For $d$ ranging from $1$ to $5$, we have the following simplicial polytopic numbers, respectively: non-negative numbers $\N$, triangular numbers $\mathrm{T}$, tetrahedral numbers $\mathrm{Te}$, pentachoron numbers $\mathrm{P}$, and hexateron numbers $\mathrm{H}$. A list of the above sets of numbers is as follows:
\begin{align*}
    \N&=\{0,1,2,3,4,5,6,7,8,9,...\},\\
    \mathrm{T}&=\{0,1,3,6,10,15,21,28,36,45,55,66,...\},\\
    \mathrm{Te}&=\{0,1,4,10,20,35,56,84,120,165,...\},\\
    \mathrm{P}&=\{0,1,5,15,35,70,126,210,330,495,715,...\},\\
    \mathrm{H}&=\{0,1,6,21,56,126,252,462,792,1287,...\}.
\end{align*}
The $n^{th}$ simplicial $d$-polytopic numbers are given by the formulae
\begin{equation}
    \brk[a]{n}_{d}=\binom{n+d-1}{d}=\frac{n^{(d)}}{d!},
\end{equation}
where $x^{(d)}=x(x+1)(x+2)\cdots(x+d-1)$ is the rising factorial.
En forma recursiva the $n^{th}$ simplicial $d$-polytopic numbers vienen dados por
\begin{equation}\label{eqn_recursive}
    \brk[a]{n+1}_{d}=\brk[a]{n}_{d}+\brk[a]{n+1}_{d-1},\ d\geq1.
\end{equation}
From Vandermonde's identity
\begin{equation}
    \binom{n+m}{r}=\sum_{k=0}^{r}\binom{m}{k}\binom{n}{r-k}
\end{equation}
we obtain the following representation for the $\s_d$-numbers
\begin{equation}\label{repre1}
    \brk[a]{n}_{d}=\sum_{k=0}^{d-1}\binom{d-1}{k}\binom{n}{k+1}.
\end{equation}
Another representation is obtained by using the numbers of Stirling of the first kind $\stirling{n}{k}$
\begin{equation}\label{repre2}
   \brk[a]{n}_d=\frac{1}{d!}\sum_{k=0}^{d}\stirling{n}{k}n^k.
\end{equation}

In \cite{cigler}, Cigler defined the $d$-Hoggatt binomial coefficients as
\begin{equation}\label{simplinomial}
    \simplinomial{n}{k}_{d}=\frac{\brk[a]{n}_{d}!}{\brk[a]{k}_{d}!\brk[a]{n-k}_{d}!},
\end{equation}
where
\begin{equation}\label{simplitorial}
    \brk[a]{n}_{d}!=\prod_{k=1}^{n}\brk[a]{k}_{d},\ n\geq1,\ \brk[a]{0}_{d}!=1.
\end{equation}
For each $n$ the $d$-Hoggatt binomials are symmetric
\begin{equation}
    \simplinomial{n}{k}_d=\simplinomial{n}{n-k}_d.
\end{equation}

The first few $d$-simplinomial coefficients are:
\begin{align*}
    \simplinomial{n}{0}_d&=1=\simplinomial{n}{n}_d,\hspace{0.3cm}\simplinomial{n}{1}_d=\simplinomial{n}{n-1}_d=\brk[a]{n}_d,\\
    \simplinomial{4}{2}_d&=\frac{\brk[a]{3}_d\brk[a]{4}_d}{\brk[a]{2}_d},\\
    \simplinomial{5}{2}_d&=\simplinomial{5}{3}_d=\frac{\brk[a]{4}_d\brk[a]{5}_d}{\brk[a]{2}_d},\\
    \simplinomial{6}{2}_d&=\simplinomial{6}{4}_d=\frac{\brk[a]{5}_d\brk[a]{6}_d}{\brk[a]{2}_d},\hspace{0.3cm}\simplinomial{6}{3}_d=\frac{\brk[a]{4}_d\brk[a]{5}_d\brk[a]{6}_d}{\brk[a]{2}_d\brk[a]{3}_d}.
\end{align*}
The entries
\begin{equation}
    \simplinomial{n}{k}_2=\frac{1}{k+1}\binom{n}{k}\binom{n+1}{k}
\end{equation}
are the Narayana numbers. 
\begin{proposition}
For $d\geq1$,
\begin{equation}\label{eqn_simplitorial1}
    \brk[a]{n}_{d}!=\frac{d^{(n)}}{d^n}\brk[a]{n}_{d-1}!
\end{equation}
\end{proposition}
\begin{proof}
From the definition of $d$-simplitorial Eq.(\ref{simplitorial})
\begin{align*}
    \brk[a]{n}_{d}!&=\prod_{k=1}^{n}\brk[a]{k}_{d}\\
    &=\prod_{k=1}^{n}\binom{k+d-1}{d}\\
    &=\prod_{k=1}^{n}\frac{k+d-1}{d}\binom{k+d-2}{d-1}\\
    &=\frac{d(d+1)\cdots(d+n-1)}{d^n}\prod_{k=1}^{n}\brk[a]{k}_{d-1}
\end{align*}
\end{proof}

\begin{proposition}
    \begin{equation}\label{eqn_simplitorial2}
        \brk[a]{n}_d!=\frac{n!(n+1)!(n+2)!\cdots(n+d-1)!}{(d!)^n0!1!2!\cdots(d-1)!}.
    \end{equation}
\end{proposition}
\begin{proof}
By iterating Eq.(\ref{eqn_simplitorial1})
\begin{align*}
    \brk[a]{n}_d!&=\frac{d^{(n)}(d-1)^{(n)}\cdots2^{(n)}1^{(n)}}{(d!)^n}.
\end{align*}
Considering that $d^{(n)}=\frac{(n+d-1)!}{(d-1)!}$, then
\begin{align*}
    \brk[a]{n}_d!=\frac{n!(n+1)!(n+2)!\cdots(n+d-1)!}{0!1!2!\cdots(d-1)!(d!)^n}.
\end{align*}
The proof is reached.
\end{proof}

\begin{proposition}
    \begin{equation}\label{eqn_sim_prod}
        \simplinomial{n}{k}_d=\prod_{i=0}^{d-1}\frac{i!(n+i)!}{(k+i)!(n-k+i)!}.
    \end{equation}
\end{proposition}
\begin{proof}
From the definition of $d$-simplinomial coefficients and Eq.(\ref{eqn_simplitorial2})
\begin{align*}
    \simplinomial{n}{k}_d&=\frac{\brk[a]{n}_d!}{\brk[a]{k}_d!\brk[a]{n-k}_d!}=\frac{\frac{1}{(d!)^n}\prod_{i=0}^{d-1}\frac{(n+i)!}{i!}}{\frac{1}{(d!)^k}\prod_{i=0}^{d-1}\frac{(k+i)!}{i!}\frac{1}{(d!)^{n-k}}\prod_{i=0}^{d-1}\frac{(n-k+i)!}{i!}}\\
    &=\prod_{i=0}^{d-1}\frac{i!(n+i)!}{(k+i)!(n-k+i)!}.
\end{align*}
\end{proof}
The results in the Eqs. (\ref{eqn_simplitorial2}) and (\ref{eqn_sim_prod}) was obtained by induction method by Cigler in \cite{cigler}. Next, we will look for an $\s_d$-analogue of Pascal’s triangle for the $d$-Hoggatt binomial coefficients. The rising factorial fulfills the relation of type  binomial
\begin{equation*}
    (a+b)^{(n)}=\sum_{k=0}^{n}\binom{n}{k}(a)^{(n-k)}(b)^{(k)}
\end{equation*}
and involves the following identity
\begin{equation}\label{eqn_sum}
    \brk[a]{a+b}_n=\sum_{k=0}^{n}\brk[a]{a}_{n-k}\brk[a]{b}_k.
\end{equation}
From here, we have the following results.
\begin{proposition}
For $d\geq2$ and for $1\leq k\leq n$, we have the following $\s_d$-analogue of the Pascal triangle
\begin{equation}\label{Sd-pascal}
    \simplinomial{n+1}{k}_{d}=f(d;n,k)\simplinomial{n}{k}_d+\simplinomial{n}{k-1}_d,
\end{equation}
where
\begin{equation}
    f(d;n,k)=1+\frac{1}{\brk[a]{n+1-k}_{d}}\sum_{i=1}^{d-1}\brk[a]{n+1-k}_{d-i}\brk[a]{k}_i
\end{equation}
\end{proposition}
\begin{proof}
From the definition of $d$-simplinomial Eq.(\ref{simplinomial})
\begin{align*}
    \simplinomial{n+1}{k}_{d}&=\frac{\brk[a]{n+1}_{d}!}{\brk[a]{k}_{d}!\brk[a]{n+1-k}_{d}!}
    =\frac{\brk[a]{n}_{d}!}{\brk[a]{k}_{d}!\brk[a]{n+1-k}_{d}!}\brk[a]{n+1}_d.
\end{align*}
By applying Eq.(\ref{eqn_sum}) with $\brk[a]{n+1}_d=\brk[a]{n+1-k+k}_d$,
\begin{align*}
    \simplinomial{n+1}{k}_d&=\frac{\brk[a]{n}_{d}!}{\brk[a]{k}_{d}!\brk[a]{n+1-k}_{d}!}\left(\brk[a]{n+1-k}_d+\brk[a]{k}_d+\sum_{i=1}^{d-1}\brk[a]{n+1-k}_{d-i}\brk[a]{k}_i\right)\\
    &=\frac{\brk[a]{n}_{d}!}{\brk[a]{k}_{d}!\brk[a]{n-k}_{d}!}+\frac{\brk[a]{n}_{d}!}{\brk[a]{k-1}_{d}!\brk[a]{n+1-k}_{d}!}\\
    &\hspace{1cm}+\frac{\brk[a]{n}_{d}!}{\brk[a]{k}_{d}!\brk[a]{n+1-k}_{d}!}\sum_{i=1}^{d-1}\brk[a]{n+1-k}_{d-i}\brk[a]{k}_i\\
    &=\simplinomial{n}{k}_d+\simplinomial{n}{k-1}_d+\frac{1}{\brk[a]{n+1-k}_{d}}\simplinomial{n}{k}_d\sum_{i=1}^{d-1}\brk[a]{n+1-k}_{d-i}\brk[a]{k}_i.    
\end{align*}
\end{proof}

When $d=2$, the Pascal triangle for the Narayana numbers is
\begin{equation}
    \simplinomial{n+1}{k}_2=\frac{n+2+k}{n+2-k}\simplinomial{n}{k}_2+\simplinomial{n}{k-1}_2.
\end{equation}
The following identity follows straightforwardly
\begin{equation}
    \binom{n+1}{k}\binom{n+2}{k}=\frac{n+2+k}{n+2-k}\binom{n}{k}\binom{n+1}{k}+\frac{k+1}{k}\binom{n}{k-1}\binom{n+1}{k-1}.
\end{equation}

\section{Simplicial $d$-polytopic calculus}

\subsection{Simplicial $d$-polytopic derivative}

\begin{definition}
Set $f(x)\in C^{d}(X)$. We define the simplicial $d$-polytopic derivative $\D_{S_{d}}f(x)$ of the function $f(x)$ by
\begin{equation}
    \D_{S_{d}}f(x)=\sum_{k=0}^{d-1}\binom{d-1}{k}\frac{1}{(k+1)!}X^{k}\D^{k+1}f(x).
\end{equation}
Some specializations are:
\begin{enumerate}
    \item For $d=1$, we obtain the usual derivative $\D_{S_{1}}=\D$.
    \item For $d=2$, we define the triangular derivative 
    \begin{equation*}
    \D_{T}=\D_{S_{2}}=\frac{1}{2}X\D^2+\D.
    \end{equation*}
    \item For $d=3$, we define the tetrahedral derivative 
    \begin{equation*}
    \D_{Te}=\D_{S_{3}}=\frac{1}{6}X^2\D^3+X\D^2+\D.
    \end{equation*}
    \item For $d=4$, we define the pentachoron derivative
    \begin{equation*}
    \D_{P}=\D_{S_{4}}=\frac{1}{24}X^3\D^4+\frac{1}{2}X^2\D^3+\frac{3}{2}X\D^2+\D.
    \end{equation*}
    \item For $d=5$, we define the hexateron derivative
    \begin{equation*}
    \D_{H}=\D_{S_{5}}=\frac{1}{120}X^4\D^5+\frac{1}{6}X^3\D^4+X^2\D^3+2X\D^2+\D.
    \end{equation*}
\end{enumerate}
\end{definition}

\begin{theorem}
Para todo par de funciones $f$ y $g$ en $C^{d}(\R)$ y para todo $c\in\R$
\begin{enumerate}
    \item $\D_{S_{d}}(f+g)(x)=(\D_{S_{d}}f)(x)+(\D_{S_{d}}g)(x)$.
    \item $\D_{S_{d}}(cf)(x)=c(\D_{S_{d}}f)(x)$.
\end{enumerate}
\end{theorem}
\begin{proof}
    
\end{proof}
\begin{theorem}
For all $n\in\N$,
\begin{equation}
    \D_{S_{d}}x^n=\brk[a]{n}_{d}x^{n-1}.
\end{equation}
\end{theorem}
\begin{proof}
If $n>d$, then
\begin{align*}
    \D_{\s_{d}}x^{n}&=\sum_{k=0}^{d-1}\binom{d-1}{k}\frac{1}{(k+1)!}X^{k}\D^{k+1}x^n\\
    &=\sum_{k=0}^{d-1}\binom{d-1}{k}\frac{1}{(k+1)!}n(n-1)(n-2)\cdots(n-k)x^{k}x^{n-k-1}\\
    &=\sum_{k=0}^{d-1}\binom{d-1}{k}\binom{n}{k+1}x^{n-1}\\
    &=\brk[a]{n}_{d}x^{n-1}.
\end{align*}
If $n\leq d$, then
\begin{align*}
    \D_{\s_d}x^n&=\sum_{k=0}^{n-1}\binom{d-1}{k}\binom{n}{k+1}x^{n-1}\\
    &=\sum_{k=0}^{d-1}\binom{d-1}{k}\binom{n}{k+1}x^{n-1}=\brk[a]{n}_dx^{n-1}
\end{align*}
since $\binom{n}{k+1}=0$ when $n\leq k\leq d-1$.
\end{proof}

\begin{theorem}
For all $k\in\N$,
    \begin{equation}
        \D_{\s_d}^kx^n=\frac{\brk[a]{n}_d!}{\brk[a]{n-k}_d!}x^{n-k}.
    \end{equation}
\end{theorem}

\begin{definition}
Set $f(x)\in C^{d}(X)$. We define a second simplicial $d$-polytopic derivative $\D_{\s^\prime_{d}}f(x)$ of the function $f(x)$ by
   \begin{equation}\label{def2}
        \D_{\s^\prime_d}=\frac{1}{d!}\sum_{k=0}^{d}\stirling{d}{k}X^{-1}(XD)^k.
    \end{equation}
\end{definition}
\begin{theorem}
For all $n\in\N$,
\begin{equation}
    \D_{\s^\prime_{d}}x^n=\brk[a]{n}_{d}x^{n-1}.
\end{equation}
\end{theorem}
\begin{proof}
From the definition of derivative $\D_{\s^\prime_d}$ Eq.(\ref{def2}) we have
    \begin{align*}
        \D_{\s^\prime_d}x^n&=\frac{1}{d!}\sum_{k=0}^{d}\stirling{d}{k}X^{-1}(XD)^kx^n\\
        &=\frac{1}{d!}\sum_{k=0}^{d}\stirling{d}{k}n^kx^{n-1}=\brk[a]{n}_dx^{n-1},
    \end{align*}
where we have used the representation for $\s_d$-numbers Eq.(\ref{repre2}).
\end{proof}

It would be very interesting to show that the definitions of $\s_d$-derivatives $\D_{\s_d}$ and $\D_{\s^\prime_d}$ are equivalent.

\begin{example}
Provided that the definitions of $\s_d$-derivatives are equivalent, then
\begin{equation}\label{eqn_KT}
    {}_{1}F_{1}(1-d;2;-x)=\frac{x^{-1}}{d!}\sum_{k=0}^{d}\stirling{d}{k}\mathrm{T}_k(x)
\end{equation}
where $_{1}F_{1}(a;b;x)$ is the Kummer confluent hypergeometric function and $\mathrm{T}_n(x)$ are the Touchard polynomials of degree $n$.
On the one side
    \begin{align*}
        \D_{S_{d}}e^{x}&=\sum_{k=0}^{d-1}\binom{d-1}{k}\frac{1}{(k+1)!}X^k\D^{k+1}e^x\\
        &=\sum_{k=0}^{d-1}\binom{d-1}{k}\frac{x^k}{(k+1)!}e^x\\
        &=e^{x}\sum_{k=0}^{d-1}\binom{d-1}{k}\frac{x^k}{(k+1)!}\\
        &=e^{x}\ _{1}F_{1}(1-d;2;-x).
    \end{align*}
On the other side
\begin{align*}
    \D_{\s^{\prime}_d}e^x&=\frac{1}{d!}\sum_{k=0}^{d}\stirling{d}{k}X^{-1}(XD)^ke^x\\
    &=\frac{x^{-1}}{d!}\sum_{k=0}^{d}\stirling{d}{k}\sum_{i=1}^{k}\stirlingII{k}{i}X^iD^ie^x\\
    &=\frac{x^{-1}e^x}{d!}\sum_{k=0}^{d}\stirling{d}{k}\sum_{i=0}^{k}\stirlingII{k}{i}x^i\\
    &=\frac{x^{-1}e^x}{d!}\sum_{k=0}^{d}\stirling{d}{k}\mathrm{T}_k(x).
\end{align*}
The identity in Eq.(\ref{eqn_KT}) is proved.
\end{example}

\begin{theorem}
For $d\geq2$
    \begin{equation}
       \D_{\s_d}(fg)=\D_{\s_d}f\cdot g+f\D_{\s_d}f+\sum_{k=1}^{d-1}\binom{d-1}{k}\frac{1}{(k+1)!}x^{k}\sum_{i=1}^{k}\binom{k+1}{i}f^{(k+1-i)}g^{(i)}.
   \end{equation}
\end{theorem}
\begin{proof}
    \begin{align*}
        \D_{\s_d}(fg)&=\sum_{k=0}^{d-1}\binom{d-1}{k}\frac{1}{(k+1)!}X^{k}\D^{k+1}(fg)\\
        &=\sum_{k=0}^{d-1}\binom{d-1}{k}\frac{1}{(k+1)!}x^{k}\sum_{i=0}^{k+1}\binom{k+1}{i}f^{(k+1-i)}g^{(i)}\\
        &=\sum_{k=0}^{d-1}\binom{d-1}{k}\frac{1}{(k+1)!}x^{k}\left(f^{(k+1)}g+fg^{(k+1)}+\sum_{i=1}^{k}\binom{k+1}{i}f^{(k+1-i)}g^{(i)}\right)\\
        &=\D_{\s_d}f\cdot g+f\D_{\s_d}f+\sum_{k=1}^{d-1}\binom{d-1}{k}\frac{1}{(k+1)!}x^{k}\sum_{i=1}^{k}\binom{k+1}{i}f^{(k+1-i)}g^{(i)}.
    \end{align*}
\end{proof}

\subsection{Simplicial $d$-polytopic binomial theorem}

\begin{definition}
For all $n\in\N$ and for $d\geq2$, we define the $\s_d$-analogue of the binomial theorem by
    \begin{equation}
        (x\oplus_{d}y)^{(n)}=
        \begin{cases}
        \sum_{k=0}^{n}\simplinomial{n}{k}_{d}x^{n-k}y^{k},&\text{ if }n>0;\\
        1,&\text{ if }n=0.
        \end{cases}
    \end{equation}
\end{definition}
From \cite{cigler2}, 
\begin{equation}
    (1\oplus_dx)^{(n)}=\sum_{k=0}^{n}\simplinomial{n}{k}_dx^k
\end{equation}
is the $d$-Hoggatt polynomial. Then we can see that $(x\oplus_dy)^{(n)}$ is the bivariate $d$-Hoggatt polynomials.
In this regard, $(x\oplus_2y)^{(n)}$ is the bivariate Narayana polynomial of order $n$.
\begin{theorem}
For all $n\in\N$, the bivariate $d$-Hoggatt polynomials holds the recurrence
    \begin{multline}
        (x\oplus_dy)^{(n+1)}=(x+y)(x\oplus_dy)^{(n)}\\
            +\sum_{k=1}^{n}\frac{1}{\brk[a]{n+1-k}_{d}}\simplinomial{n}{k}_d\sum_{i=1}^{d-1}\brk[a]{n+1-k}_{d-i}\brk[a]{k}_ix^{n+1-k}y^k.
    \end{multline}
\end{theorem}
\begin{proof}
From Pascal triangle Eq.(\ref{Sd-pascal}),
\begin{align*}
    (x\oplus_dy)^{(n+1)}&=\sum_{k=0}^{n+1}\simplinomial{n+1}{k}_{d}x^{n+1-k}y^{k}=x^{n+1}+y^{n+1}+\sum_{k=1}^{n}\simplinomial{n+1}{k}_{d}x^{n+1-k}y^{k}\\
    &=x^{n+1}+y^{n+1}+\sum_{k=1}^{n}\simplinomial{n}{k}_dx^{n+1-k}y^k+\sum_{k=1}^{n}\simplinomial{n}{k-1}_dx^{n+1-k}y^k\\
    &\hspace{1cm}+\sum_{k=1}^{n}\frac{1}{\brk[a]{n+1-k}_{d}}\simplinomial{n}{k}_d\sum_{i=1}^{d-1}\brk[a]{n+1-k}_{d-i}\brk[a]{k}_ix^{n+1-k}y^k.
\end{align*}
Collect terms and factorize
\begin{align*}
    (x\oplus_dy)^{(n+1)}&=x\sum_{k=0}^{n}\simplinomial{n}{k}_dx^{n-k}y^k+y\sum_{k=0}^{n}\simplinomial{n}{k}_dx^{n-k}y^{k}\\
    &\hspace{1cm}+\sum_{k=1}^{n}\frac{1}{\brk[a]{n+1-k}_{d}}\simplinomial{n}{k}_d\sum_{i=1}^{d-1}\brk[a]{n+1-k}_{d-i}\brk[a]{k}_ix^{n+1-k}y^k\\
    &=(x+y)(x\oplus_dy)^{(n)}\\
    &\hspace{2cm}+\sum_{k=1}^{n}\frac{1}{\brk[a]{n+1-k}_{d}}\simplinomial{n}{k}_d\sum_{i=1}^{d-1}\brk[a]{n+1-k}_{d-i}\brk[a]{k}_ix^{n+1-k}y^k.
\end{align*}
\end{proof}
When $d=2$,
\begin{multline}
    (x\oplus_2y)^{(n+1)}=(x+y)(x\oplus_2y)^{(n)}\\
    +\sum_{k=1}^{n}\frac{2k}{(k+1)(n+2-k)}\binom{n}{k}\binom{n+1}{k}x^{n+1-k}y^k.
\end{multline}

\begin{theorem}
For all $n\in\N$, the $\s_d$-derivative of the bivariate Hoggatt polynomials is
    \begin{equation}
        \D_{S_{d}}(x\oplus_{d}a)^{(n)}=\brk[a]{n}_{d}(x\oplus_{d}a)^{(n-1)}.
    \end{equation}
\end{theorem}
\begin{proof}
From the definition of $\s_d$-derivative
    \begin{align*}
        \D_{S_{d}}(x\oplus_{d}a)^{(n)}&=\D_{S_{d}}\left(\sum_{k=0}^{n-1}\simplinomial{n}{k}_{d}x^{n-k}a^{k}+a^n\right)\\
        &=\sum_{k=0}^{n-1}\frac{\brk[a]{n}_{d}!}{\brk[a]{k}_{d}!\brk[a]{n-k}_{d}!}\brk[a]{n-k}_{d}x^{n-k-1}a^{k}\\
        &=\sum_{k=0}^{n-1}\frac{\brk[a]{n}_{d}!}{\brk[a]{k}_{d}!\brk[a]{n-k-1}_{d}!}x^{n-k-1}a^{k}\\
        &=\brk[a]{n}_{d}\sum_{k=0}^{n-1}\frac{\brk[a]{n-1}_{d}!}{\brk[a]{k}_{d}!\brk[a]{n-1-k}_{d}!}x^{n-k-1}a^{k}\\
        &=\brk[a]{n}_{d}(x\oplus_{d}a)^{(n-1)}.
    \end{align*}
\end{proof}

\subsection{Simplicial $d$-polytopic-type exponential function}

\begin{definition}
Define the simplicial $d$-polytopic-type exponential function (shortly $\s_d$-exponential function) 
\begin{equation}
    \exp_{d}(x)=\sum_{n=0}^{\infty}\frac{x^n}{\brk[a]{n}_{d}!}.
\end{equation}
\end{definition}
From Eq.(\ref{eqn_simplitorial2})
\begin{equation}
    \exp_d(x)=\prod_{i=0}^{d-1}i!\sum_{n=0}^{\infty}\frac{(d!x)^n}{n!(n+1)!(n+2)!\cdots(n+d-1)!}.
\end{equation}
As $(m)_n=(n+m-1)!/(m-1)!$, then
\begin{equation}\label{exp_hyper}
    \exp_d(x)=\sum_{n=0}^{\infty}\frac{(d!x)^n}{n!(2)_n(3)_n\cdots(d)_n}={}_{0}F_{d-1}(;2,3,\ldots,d;d!x),
\end{equation}
where
\begin{equation}
    {}_{p}F_{q}(a_1,\ldots,a_p;b_1,\ldots,b_q;x)=\sum_{n=0}^{\infty}\frac{(a_{1})_{n}(a_{2})_{n}\cdots(a_{r})_{n}}{(b_{1})_{n}(b_{2})_{n}\cdots(b_{s})_{n}}\frac{x^n}{n!}
\end{equation}
is the generalized hypergeometric function. Some specialization of $\s_d$-exponential functions are
\begin{align}
    \e^x=\exp_{1}(x)&={}_{0}F_{0}(;;x),\\
    \exp_{\mathrm{T}}(x)=\exp_{2}(x)&={}_{0}F_{1}(;2;2x),\\
    \exp_{\mathrm{Te}}(x)=\exp_{3}(x)&={}_{0}F_{2}(;2,3;6x),\\
    \exp_{\mathrm{P}}(x)=\exp_{4}(x)&={}_{0}F_{3}(;2,3,4;24x),\\
    \exp_{\mathrm{H}}(x)=\exp_{5}(x)&={}_{0}F_{4}(;2,3,4,5;120x).
\end{align}

\begin{theorem}
The $\s_d$-exponential function $\exp_{d}(z)$ is an entire function for all $d\in\N$.
\end{theorem}
\begin{proof}
The convergence of $\exp_d(x)$ follows from the convergence of the generalized hypergeometric function ${}_{0}F_{d-1}$.   
\end{proof}
When $d=2$, the initial value problem
\begin{equation}
    \D_{\mathrm{T}}y=y,\hspace{0.5cm}y(0)=1
\end{equation}
is equivalent to the problem
\begin{equation*}
    xy^{\prime\prime}+2y^{\prime}-2y=0,\hspace{0.5cm}y(0)=1
\end{equation*}
the solution to which is
\begin{equation*}
    y(x)=\frac{I_{1}(2\sqrt{2x})}{\sqrt{2x}},
\end{equation*}
where $I_{n}(x)$ is the first order modified Bessel function. Then
\begin{equation}
    \exp_{\mathrm{T}}(x)=\frac{I_{1}(2\sqrt{2x})}{\sqrt{2x}}.
\end{equation}





\begin{proposition}
\begin{equation}\label{eqn_prod_exp}
    \exp_{d}(x)\exp_{d}(y)=\exp_{d}(x\oplus_{d}y)
\end{equation}
\end{proposition}
\begin{proof}
    \begin{align*}
        \exp_d(x)\exp_d(y)&=\left(\sum_{n=0}^{\infty}\frac{x^n}{\brk[a]{n}_d!}\right)\left(\sum_{n=0}^{\infty}\frac{y^n}{\brk[a]{n}_d!}\right)\\
        &=\sum_{n=0}^{\infty}\left(\sum_{k=0}^{n}\simplinomial{n}{k}_dx^{n-k}y^k\right)\frac{1}{\brk[a]{n}_d!}=\sum_{n=0}^{\infty}\frac{(x\oplus_dy)^{(n)}}{\brk[a]{n}_d!}\\
        &=\exp_d(x\oplus_dy).
    \end{align*}
\end{proof}
Then,
\begin{equation}
    \frac{I_1(2\sqrt{2xt})}{\sqrt{2xt}}\frac{I_1(2\sqrt{2yt})}{\sqrt{2yt}}=\sum_{n=0}^{\infty}(x\oplus_2y)^{(n)}\frac{t^n}{\brk[a]{n}_d!}
\end{equation}
is the generating function of the bivariate Narayana polynomials.

\subsection{Simplicial $d$-polytopic-type hypergeometric functions}

\begin{definition}
For a positive real number $a$, the simplicial $d$-polytopic-type Pochhammer symbol is defined by
\begin{equation}
    (a)_{d,n}=
    \begin{cases}
        1,&\text{ if }n=0;\\
        \brk[a]{a}_d\brk[a]{a+1}_d\brk[a]{a+2}_d\cdots\brk[a]{a+n-1}_d,&\text{ if }n>0.
    \end{cases}
\end{equation}
\end{definition}
The simplicial $d$-polytopic-type Pochhammer symbol is directly related to $d$-simplitorial:
\begin{equation}
    (1)_{d,n}=\brk[a]{n}_d!
\end{equation}

\begin{definition}
For real numbers $a_1,a_2,\ldots,a_r,b_1,b_2,\ldots,b_s$ we define the simplicial $d$-polytopic-type hypergeometric series $_{r}\sigma_{s}^d$ (shortly $\s_d$-hypergeometric series) as
    \begin{align}
        _{r}\sigma^d_{s}(a_1,a_2,\ldots,a_r;b_1,b_2,\ldots,b_s;z)&\equiv{}_{r}\sigma^d_{s}\left(
    \begin{array}{c}
         a_{1},a_{2},\ldots,a_{r} \\
         b_{1},\ldots,b_{s}
    \end{array}
    ;q,z
    \right)\nonumber\\
    &=\sum_{n=0}^{\infty}\frac{(a_{1})_{d,n}(a_{2})_{d,n}\cdots(a_{r})_{d,n}}{(b_{1})_{d,n}(b_{2})_{d,n}\cdots(b_{s})_{d,n}}\frac{z^n}{\brk[a]{n}_d!},
    \end{align}
where it is assumed that $b_{i}\neq0,-1,-2,-3,\ldots$, so that no zero factors appear in the denominators of the terms of the series.
\end{definition}
When $d=1$
\begin{equation}
    _{r}\sigma^1_{s}(a_1,a_2,\ldots,a_r;b_1,b_2,\ldots,b_s;z)={}_{r}F_{s}(a_1,a_2,\ldots,a_r;b_1,b_2,\ldots,b_s;z).
\end{equation}
A special cases is the $\s_d$-exponential function 
\begin{equation}
    \exp_{d}(x)={}_{0}\sigma^d_{0}(-;-;x)={}_{0}F_{d-1}(;2,3,\ldots,d;d!x).
\end{equation}

We write
\begin{equation}
    \mathrm{T}_{d,n}(x)=\sum_{k=0}^{n}\frac{x^k}{\brk[a]{k}_d!}
\end{equation}
for the $n$-th partial sum of the $\s_d$-exponential function, which is a polynomial of degree $n$. Note that
\begin{equation}
    \exp_d(x)-\mathrm{T}_{d,n-1}(x)=\frac{x^n}{\brk[a]{n}_d!}\sum_{k=0}^{\infty}\frac{x^k}{(n+1)_{d,k}}=\frac{x^n}{\brk[a]{n}_{d}!}{}_{1}\sigma^d_{1}(1;n+1;x).
\end{equation}
Then from Eq.(\ref{exp_hyper})
\begin{equation}
    {}_{1}\sigma^d_{1}(1;n+1;x)=\frac{\brk[a]{n}_d!}{x^n}\left({}_{0}F_{d-1}(;2,3,\ldots,d;d!x)-\mathrm{T}_{d,n-1}(x)\right).
\end{equation}
\section{$\s_{d}$-Hypergeometric Bernoulli numbers and polynomials}

\begin{definition}
The simplicial $d$-polytopic-type hypergeometric Bernoulli numbers (shortly $\s_d$-hypergeometric Bernoulli) are defined by the generating function
\begin{equation}\label{hyp_Ber_num}
    \frac{t^m/\brk[a]{m}_d!}{\exp_{d}(t)-\mathrm{T}_{d,m-1}(t)}=\frac{1}{_{1}\sigma^d_{1}(1;m+1;t)}=\sum_{n=0}^{\infty}B_{d,n}(m)\frac{t^n}{\brk[a]{n}_{d}!}.
\end{equation}
\end{definition}

\begin{theorem}
    \begin{equation}
    B_{d,n}(m)=\sum_{i=1}^{n}\sum_{k_{1}+k_{2}+\cdots+k_{i}=n}\frac{(-1)^i\brk[a]{n}_d!}{(m+1)_{d,k_1}(m+1)_{d,k_2}\cdots(m+1)_{d,k_i}}.
\end{equation}
\end{theorem}
\begin{proof}
Employing a method given by Jordan \cite{jordan}, we can give an explicit expression for $B_{d,n}$. Indeed, we can write 
\begin{align*}
    \frac{1}{{}_{1}\sigma^d_{1}(1;m+1;t)}=\frac{1}{1+\left({}_{1}\sigma^d_{1}(1;m+1;t)-1\right)}&=\sum_{n=0}^{\infty}(-1)^n\left({}_{1}\sigma^d_{1}(1;m+1;t)-1\right)^n\\
    &=1+\sum_{n=1}^{\infty}(-1)^n\left(\sum_{k=1}^{\infty}\frac{t^k}{(m+1)_{d,k}}\right)^n.
\end{align*}
Thus
\begin{equation}
    B_{d,n}(m)=\sum_{i=1}^{n}\sum_{k_{1}+k_{2}+\cdots+k_{i}=n}\frac{(-1)^i\brk[a]{n}_d!}{(m+1)_{d,k_1}(m+1)_{d,k_2}\cdots(m+1)_{d,k_i}}.
\end{equation}
\end{proof}

\begin{definition}
The simplicial $d$-polytopic-type hypergeometric Bernoulli polynomials $B_{d,n}(m;x)$, or $\s_{d}$-hypergeometric Bernoulli, are defined by the generating function
\begin{equation}\label{hyp_Ber_poly}
    \frac{t^m\exp_{d}(xt)/\brk[a]{m}_d!}{\exp_{d}(t)-\mathrm{T}_{d,m-1}(t)}=\frac{\exp_{d}(xt)}{{}_{1}\sigma^d_{1}(1;m+1;t)}=\sum_{n=0}^{\infty}B_{d,n}(m;x)\frac{t^n}{\brk[a]{n}_{d}!}.
\end{equation}
\end{definition}
By comparing Eqs. (\ref{hyp_Ber_num}) and (\ref{hyp_Ber_poly}), we can write $B_{d,n}(m)=B_{d,n}(m;0)$.

\begin{theorem}
The $\s_d$-hypergeometric Bernoulli polynomials have the following representation
    \begin{equation}
        B_{d,n}(m;x)=\sum_{k=0}^{n}\simplinomial{n}{k}_{d}B_{k,d}(m)x^{n-k}.
    \end{equation}
\end{theorem}
\begin{proof}
From definition of $\s_d$-hypergeometric Bernoulli polynomials Eq.(\ref{hyp_Ber_poly}), we have
    \begin{align*}
        \sum_{n=0}^{\infty}B_{d,n}(m;x)\frac{t^n}{\brk[a]{n}_{d}!}&=\frac{1}{{}_{1}\sigma^d_{1}(1;m+1;t)}\exp_{d}(xt)\\
        &=\left(\sum_{n=0}^{\infty}B_{d,n}(m)\frac{t^n}{\brk[a]{n}_{d}!}\right)\left(\sum_{n=0}^{\infty}x^n\frac{t^n}{\brk[a]{n}_{d}!}\right)\\
        &=\sum_{n=0}^{\infty}\left(\sum_{k=0}^n\simplinomial{n}{k}_dB_{k,d}(m)x^{n-k}\right)\frac{t^n}{\brk[a]{n}_{d}!}.
    \end{align*}
The result is obtained by identifying the coefficients of $t^n$ on both sides of the previous equality.
\end{proof}

\begin{theorem}
For all $n\geq1$,
    \begin{equation}
        \D_{\s_{d}}B_{d,n}(m;x)=\brk[a]{n}_{d}B_{d,n-1}(m;x).
    \end{equation}
\end{theorem}
\begin{proof}
From definition of $\s_d$-hypergeometric Bernoulli polynomials Eq.(\ref{hyp_Ber_poly}), we have
    \begin{align*}
        \sum_{n=0}^{\infty}\D_{\s_d}B_{d,n}(m;x)\frac{t^n}{\brk[a]{n}_{d}!}&=\D_{\s_d}\left(\sum_{n=0}^{\infty}B_{d,n}(m;x)\frac{t^n}{\brk[a]{n}_{d}!}\right)\\
        &=\D_{\s_d}\left\{\frac{\exp_{d}(xt)}{{}_{1}\sigma^d_{1}(1;m+1;t)}\right\}\\
        &=\frac{t\exp_{d}(xt)}{{}_{1}\sigma^d_{1}(1;m+1;t)}\\
        &=\brk[a]{n}_{d}\sum_{n=0}^{\infty}B_{d,n-1}(m;x)\frac{t^n}{\brk[a]{n}_{d}!}.
    \end{align*}
The result follows by identifying the coefficients of $t^n$.
\end{proof}

\begin{theorem}
The translation formula holds
    \begin{equation}
        B_{d,n}(m;x\oplus_dy)=\sum_{k=0}^{n}\simplinomial{n}{k}_{d}B_{d,n}(m;x)y^{n-k}.
    \end{equation}
\end{theorem}
\begin{proof}
From the product of $\s_d$-exponential functions Eq.(\ref{eqn_prod_exp}), we have
    \begin{align*}
        \sum_{n=0}^{\infty}B_{d,n}(m;x\oplus_dy)\frac{t^n}{\brk[a]{n}_d!}&=\frac{\exp_{d}(x\oplus_dy)}{{}_{1}\sigma^d_{1}(1;m+1;t)}=\frac{\exp_{d}(xt)\exp_{d}(yt)}{{}_{1}\sigma^d_{1}(1;m+1;t)}\\
        &=\left(\sum_{n=0}^{\infty}B_{d,n}(m;x)\frac{t^n}{\brk[a]{n}_{d}!}\right)\left(\sum_{n=0}^{\infty}y^n\frac{t^n}{\brk[c]{n}_{d}!}\right)\\
        &=\sum_{n=0}^{\infty}\left(\sum_{k=0}^{n}\simplinomial{n}{k}_{d}B_{d,n}(m;x)y^{n-k}\right)\frac{t^n}{\brk[a]{n}_{d}!}.
    \end{align*}
The result is obtained by identifying the coefficients of $t^n$ on both sides of the previous equality.   
\end{proof}

\begin{theorem}
For all $0\leq n\leq m-1$
\begin{equation}\label{eqn_rep_xmas1}
    B_{d,n}(m;x\oplus_d1)=\sum_{k=0}^{n}\simplinomial{n}{k}_dB_{d,n-k}(m;x)
\end{equation}
and for $n\geq m$
\begin{equation}\label{eqn_rep_xmas1_pow}
    B_{d,n}(m;x\oplus_d1)-\sum_{k=0}^{m-1}\simplinomial{n}{k}_dB_{d,n-k}(m;x)=\simplinomial{n}{m}_dx^{n-m}
\end{equation}
\end{theorem}
\begin{proof}
We have that
    \begin{align*}
        &\sum_{n=0}^{\infty}B_{d,n}(m;x\oplus_d1)\frac{t^n}{\brk[a]{n}_d!}-\mathrm{T}_{d,m-1}(t)\sum_{n=0}^{\infty}B_{d,n}(m;x)\frac{t^n}{\brk[a]{n}_d!}\\
        &\hspace{3cm}=\frac{t^m\exp_{d}(xt)\exp_d(t)/\brk[a]{m}_d!}{\exp_d(t)-\mathrm{T}_{d,m-1}(t)}-\frac{t^m\mathrm{T}_{d,m}(t)\exp_d(xt)/\brk[a]{m}_d!}{\exp_d(t)-\mathrm{T}_{d,m-1}(t)}\\
        &\hspace{3cm}=\frac{t^m}{\brk[a]{m}_d!}\exp_d(xt)=\frac{1}{\brk[a]{m}_d!}\D_{\s_d}^m\exp_d(xt).
    \end{align*}
As
\begin{align*}
    \mathrm{T}_{d,m}(t)\sum_{n=0}^{\infty}B_{d,n}(m;x)\frac{t^n}{\brk[a]{n}_d!}&=\sum_{n=0}^{m-1}\sum_{k=0}^{n}\simplinomial{n}{k}_dB_{d,n-k}(m;x)\frac{t^n}{\brk[a]{n}_d!}\\
    &\hspace{1cm}+\sum_{n=m}^{\infty}\sum_{k=0}^{m-1}\simplinomial{n}{k}_dB_{d,n-k}(m;x)\frac{t^n}{\brk[a]{n}_d!},
\end{align*}
then for $0\leq n\leq m-1$
\begin{equation*}
    B_{d,n}(m;x\oplus_d1)=\sum_{k=0}^{n}\simplinomial{n}{k}_dB_{d,n-k}(m;x)
\end{equation*}
and for $n\geq m$
\begin{equation*}
    B_{d,n}(m;x\oplus_d1)-\sum_{k=0}^{m-1}\simplinomial{n}{k}_dB_{d,n-k}(m;x)=\simplinomial{n}{m}_dx^{n-m}.
\end{equation*}
\end{proof}

\begin{corollary}
For all $0\leq n\leq m-1$
\begin{equation}
    B_{d,n}(m;1)=\sum_{k=0}^{n}\simplinomial{n}{k}_dB_{d,n-k}(m)
\end{equation}
for $n=m$,
\begin{equation}
    B_{d,n}(n;1)=\sum_{k=0}^{n-1}\simplinomial{n}{k}_dB_{d,n-k}(n)+1,
\end{equation}
and for $n>m$
\begin{equation}
    B_{d,n}(m;1)=\sum_{k=0}^{m-1}\simplinomial{n}{k}_dB_{d,n-k}(m).
\end{equation}
\end{corollary}
\begin{proof}
Take $x=0$ in Eqs. (\ref{eqn_rep_xmas1}) and (\ref{eqn_rep_xmas1_pow}).    
\end{proof}

\begin{itemize}
    \item Cases $m=1$. 
    \begin{align*}
        B_{d,0}(1;1)&=B_{d,0}(1),\\
        B_{d,n}(1;1)&=B_{d,n}(1),\hspace{0.5cm}n\geq1.
    \end{align*}
    \item Cases $m=2$. 
    \begin{align*}
        B_{d,0}(2;1)&=B_{d,0}(2),\\
        B_{d,1}(2;1)&=B_{d,1}(2)+B_{d,0}(2),\\
        B_{d,n}(2;1)&=B_{d,n}(2)+\brk[a]{n}_dB_{d,n-1}(2),\hspace{0.5cm}n\geq2.
    \end{align*}
    \item Cases $m=3$. 
    \begin{align*}
        B_{d,0}(3;1)&=B_{d,0}(3),\\
        B_{d,1}(3;1)&=B_{d,1}(3)+B_{d,0}(3),\\
        B_{d,2}(3;1)&=B_{d,2}(3)+\brk[a]{2}_dB_{d,1}(3)+B_{d,0}(3),\\
        B_{d,n}(3;1)&=B_{d,n}(3)+\brk[a]{n}_dB_{d,n-1}(3)+\frac{\brk[a]{n-1}_d\brk[a]{n}_d}{\brk[a]{2}_d}B_{d,n-2}(3),\hspace{0.5cm}n\geq3.
    \end{align*}
\end{itemize}

\begin{theorem}
The inversion formula holds
    \begin{equation}\label{eqn_inver}
        x^n=\sum_{k=0}^{n}\frac{\brk[a]{n}_d!}{(m+1)_{d,k}\brk[a]{n-k}_d}B_{d,n-k}(m;x).
    \end{equation}
\end{theorem}
\begin{proof}
From the definition of the $\s_d$-hypergeometric Bernoulli polynomials Eq.(\ref{hyp_Ber_poly}), we have
    \begin{align*}
        \sum_{n=0}^{\infty}x^n\frac{t^n}{\brk[a]{n}_d!}&=\exp_d(xt)={}_{1}\sigma^d_{1}\left(\begin{array}{c}
             1  \\
             m+1 
        \end{array};t\right)\left(\sum_{n=0}^{\infty}B_{d,n}(m;x)\frac{t^n}{\brk[a]{n}_d!}\right)\\
        &=\left(\sum_{n=0}^{\infty}\frac{t^n}{(m+1)_{d,n}}\right)\left(\sum_{n=0}^{\infty}B_{d,n}(m;x)\frac{t^n}{\brk[a]{n}_d!}\right)\\
        &=\sum_{n=0}^{\infty}\left(\sum_{k=0}^{n}\frac{\brk[a]{n}_d!}{(m+1)_{d,k}\brk[a]{n-k}_d}B_{d,n-k}(m;x)\right)\frac{t^n}{\brk[a]{n}_d!}.
    \end{align*}
The result is obtained by identifying the coefficients of $t^n$ on both sides of the previous equality.
\end{proof}

\begin{corollary}
The following equation applies
    \begin{equation}
        \sum_{k=0}^{n}\frac{\brk[a]{n}_d!}{(m+1)_{d,k}\brk[a]{n-k}_d}B_{d,n-k}(m)=
        \begin{cases}
            1,&\text{ if }n=0;\\
            0,&\text{ if }n\neq0.
        \end{cases}
    \end{equation}
\end{corollary}
\begin{proof}
Take $x=0$ in Eq.(\ref{eqn_inver}).    
\end{proof}

\section{$\s_d$-Bernoulli numbers and polynomials}

\begin{definition}
The $\s_d$-Bernoulli numbers are defined by
\begin{equation}\label{Ber_num}
    \frac{t}{\exp_{d}(t)-1}=\sum_{n=0}^{\infty}B_{d,n}\frac{t^n}{\brk[a]{n}_{d}!}.
\end{equation}
\end{definition}
The generating function of the Triangular-Bernoulli $B_{\mathrm{T},n}$ is
\begin{equation}
    \frac{t\sqrt{2t}}{I_{1}(2\sqrt{2t})-\sqrt{2t}}=\sum_{n=0}^{\infty}B_{\mathrm{T},n}\frac{(2t)^n}{n!(n+1)!}.
\end{equation}

\begin{theorem}
    \begin{equation}
        B_{d,n}=\sum_{i=1}^{n}\sum_{k_{1}+k_{2}+\cdots+k_{i}=n}\frac{(-1)^i\brk[a]{n}_d!}{\brk[a]{k_1+1}_d!\brk[a]{k_2+1}_d!\cdots\brk[a]{k_i+1}_d!}.
    \end{equation}
\end{theorem}

The following are the first few values of $B_{d,n}$:
\begin{align*}
    B_{d,0}&=1,\hspace{0.2cm}B_{d,1}=-\frac{1}{\brk[a]{2}_d},\hspace{0.2cm}B_{d,2}=\frac{\brk[a]{3}_d-\brk[a]{2}_d}{\brk[a]{2}_d\brk[a]{3}_d},\\
    B_{d,3}&=-\frac{\brk[a]{3}_d\brk[a]{4}_d-2\brk[a]{2}_d\brk[a]{4}+\brk[a]{2}^{2}_d}{\brk[a]{2}^2_d\brk[a]{4}_d},\\
    B_{d,4}&=-\frac{1}{\brk[a]{5}_d}+\frac{\brk[a]{4}_d}{\brk[a]{2}_d\brk[a]{3}_d}+\frac{2}{\brk[a]{2}_d}-\frac{3\brk[a]{4}_d}{\brk[a]{2}_d^2}+\frac{\brk[a]{4}_d!}{\brk[a]{2}_d^4}.
\end{align*}
Some specialization of $\s_d$-Bernoulli numbers are: 
\begin{itemize}
    \item The Triangular-Bernoulli numbers:
    \begin{align*}
        B_{\mathrm{T},0}&=1,\hspace{0.2cm}B_{\mathrm{T},1}=-\frac{1}{3},\hspace{0.2cm}B_{\mathrm{T},2}=\frac{1}{2},\hspace{0.2cm}B_{\mathrm{T},3}=-\frac{1}{10},\hspace{0.2cm}B_{\mathrm{T},4}=\frac{2}{45}.
    \end{align*}
    \item The Tetrahedral-Bernoulli numbers:
    \begin{align*}
        B_{\mathrm{Te},0}&=1,\hspace{0.2cm}B_{\mathrm{Te},1}=-\frac{1}{4},\hspace{0.2cm}B_{\mathrm{Te},2}=\frac{3}{20},\hspace{0.2cm}B_{\mathrm{Te},3}=-\frac{7}{40},\hspace{0.2cm}B_{\mathrm{Te},4}=\frac{97}{280}.
    \end{align*}
    \item The Pentachoron-Bernoulli numbers:
    \begin{align*}
        B_{\mathrm{P},0}&=1,\hspace{0.2cm}B_{\mathrm{P},1}=-\frac{1}{5},\hspace{0.2cm}B_{\mathrm{P},2}=\frac{2}{15},\hspace{0.2cm}B_{\mathrm{P},3}=-\frac{8}{35},\hspace{0.2cm}B_{\mathrm{P},4}=\frac{2237}{210}.
    \end{align*}
    \item The Hexateron-Bernoulli numbers:
    \begin{align*}
        B_{\mathrm{H},0}&=1,\hspace{0.2cm}B_{\mathrm{H},1}=-\frac{1}{6},\hspace{0.2cm}B_{\mathrm{H},2}=\frac{5}{42},\hspace{0.2cm}B_{\mathrm{H},3}=-\frac{15}{56},\hspace{0.2cm}B_{\mathrm{H},4}=\frac{1755}{1334}.
    \end{align*}
\end{itemize}

\begin{definition}
The Simplicial-Polytopic-type Bernoulli polynomials $B_{d,n}(x)$, or $\s_{d}$-Bernoulli, are defined by the generating function
\begin{equation}\label{Ber_poly}
    \frac{t\exp_{d}(xt)}{\exp_{d}(t)-1}=\sum_{n=0}^{\infty}B_{d,n}(x)\frac{t^n}{\brk[a]{n}_{d}!}.
\end{equation}
\end{definition}
The generating function of the Triangular-Bernoulli polynomials is
\begin{equation}
    \frac{tI_{1}(2\sqrt{2t})}{I_{1}(2\sqrt{2t})-\sqrt{2t}}\frac{I_1(2\sqrt{2tx})}{\sqrt{2tx}}=\sum_{n=0}^{\infty}B_{\mathrm{T},n}(x)\frac{(2t)^n}{n!(n+1)!}
\end{equation}
The $\s_d$-Bernoulli polynomials have the representation
    \begin{equation}
        B_{d,n}(x)=\sum_{k=0}^{n}\simplinomial{n}{k}_{d}B_{k,d}x^{n-k}.
    \end{equation}
The first few $\s_d$-Bernoulli polynomials are:
\begin{align*}
    B_{0,d}(x)&=1,\hspace{0.3cm}B_{1,d}(x)=x-\frac{1}{\brk[a]{2}_d},\hspace{0.3cm}B_{2,d}(x)=x^2-x+\frac{\brk[a]{3}_d-\brk[a]{2}_d}{\brk[a]{2}_d\brk[a]{3}_d},\\
    B_{3,d}(x)&=x^3-\frac{\brk[a]{3}_d}{\brk[a]{2}_d}x^2+\frac{\brk[a]{3}_d\brk[a]{2}_d}{\brk[a]{2}_d}x-\frac{\brk[a]{3}_d\brk[a]{4}_d-2\brk[a]{2}_d\brk[a]{4}+\brk[a]{2}^{2}_d}{\brk[a]{2}^2_d\brk[a]{4}_d},\\
    B_{4,d}(x)&=x^4-\frac{\brk[a]{4}_d}{\brk[a]{2}_d}x^3+\frac{\brk[a]{4}_d(\brk[a]{3}_d-\brk[a]{2}_d)}{\brk[a]{2}_d^2}x^2\\
    &\hspace{1cm}-\frac{\brk[a]{3}_d\brk[a]{4}_d-2\brk[a]{2}_d\brk[a]{4}+\brk[a]{2}^{2}_d}{\brk[a]{2}^2_d\brk[a]{4}_d}x+\\
    &\hspace{2cm}\left(-\frac{1}{\brk[a]{5}_d}+\frac{\brk[a]{4}_d}{\brk[a]{2}_d\brk[a]{3}_d}+\frac{2}{\brk[a]{2}_d}-\frac{3\brk[a]{4}_d}{\brk[a]{2}_d^2}+\frac{\brk[a]{4}_d!}{\brk[a]{2}_d^4}\right).
\end{align*}
For all $n\geq1$,
    \begin{equation}
        \D_{\s_{d}}B_{d,n}(x)=\brk[a]{n}_{d}B_{n-1,d}(x).
    \end{equation}
The translation formula holds
    \begin{equation}
        B_{d,n}(x\oplus_dy)=\sum_{k=0}^{n}\simplinomial{n}{k}_{d}B_{d,n}(x)y^{n-k}.
    \end{equation}
For all $n\geq1$
    \begin{equation}\label{eqn_diff}
        B_{d,n}(x\oplus_d1)-B_{d,n}(x)=\brk[a]{n}_dx^{n-1}.
    \end{equation}
If $x=0$ in Eq.(\ref{eqn_diff}), then $B_{d,n}(1)=B_{d,n}(0)=B_{d,n}$. The inversion formula holds
    \begin{equation}
        x^n=\sum_{k=0}^n\simplinomial{n}{k}_d\frac{B_{n-k,d}(x)}{\brk[a]{k+1}}.
    \end{equation}

\end{document}